\documentclass[preprint,12pt]{elsarticle}
\usepackage{graphicx} 
\usepackage{amssymb}
\usepackage{amsthm}
\usepackage{graphics}
\usepackage{epsfig}
\usepackage{amssymb}
\usepackage{hyperref}
\usepackage{graphicx}
\usepackage{subfigure}
\usepackage{float}
\usepackage{color}
\usepackage[dvipsnames]{xcolor}
\usepackage[svgnames]{xcolor}
\usepackage[x11names]{xcolor}
\usepackage{tikz-network}
\usepackage{tikz}
\usepackage{tikz,pgfplots}
\usetikzlibrary{positioning, arrows.meta}
\usetikzlibrary{patterns}
\usepackage[ruled,vlined]{algorithm2e}
\usepackage{caption}
\DeclareCaptionLabelFormat{cont}{#1~#2\alph{ContinuedFloat}}
\usepackage{fullpage}
\usepackage{amsmath}
\usepackage{amsfonts}
\usepackage{enumerate}
\usepackage{enumitem} 
\usepackage{tikz}
\usepackage{array}
\usepackage{bigdelim}
\usepackage{nicematrix}
\usepackage{mathtools}
\usepackage{blkarray, bigstrut}
\usepackage{gauss}
\usepackage{xcolor}
\usepackage{tabularray}
\UseTblrLibrary{booktabs}

\newcommand{\pf}{\noindent {\bf Proof: }}

\newtheorem*{theoremaux}{Theorem \theoremauxnum}
\gdef\theoremauxnum{1}

\newtheorem{lemma}{\bf Lemma}[section]
\newtheorem{theorem}{\bf Theorem}[section]

\newtheorem{proposition}[lemma]{\bf Proposition}
\newtheorem{corollary}[lemma]{\bf Corollary}
\newtheorem{definition}{\bf Definition}[section]
\newtheorem{remark}{\bf Remark}[section]

\journal{~}

\begin{document}

\begin{frontmatter}
\title{{Cayley colour integral groups}}





 \author[1]{Sauvik Poddar\corref{cor1}}
\ead{sauvikpoddar1997@gmail.com}

\author[1]{Angsuman Das}
\ead{angsuman.maths@presiuniv.ac.in}

\address[1]{Department of Mathematics, Presidency University, 86/1 College Street, Kolkata 700073, India\\}
\cortext[cor1]{Corresponding author}

%

\begin{abstract}
A finite group $G$ is said to be Cayley integral if every undirected Cayley graph $\operatorname{Cay}(G,S)$ on $G$ is integral. In this paper, we introduce three natural extensions of this concept; namely as: Cayley colour integral, $\mathfrak{F}$-Cayley colour integral and normal Cayley integral groups. We characterize the first two families in its entirety. The last family of groups is shown to be coinciding with inverse semi-rational groups introduced by Chillag and Dolfi, thereby providing an alternative characterization for the same. We also establish an inclusion hierarchy among these families.
\end{abstract}

\end{frontmatter}

\section{Introduction}

All the groups and graphs considered here are finite. 
Let $G$ be a group with identity $1$ and $S\subseteq G\setminus\lbrace{1}\rbrace$ be an inverse-closed subset, i.e., $S=S^{-1}$, where $S^{-1}\coloneqq\lbrace{s^{-1}~|~s\in S}\rbrace$. A \textit{Cayley graph} $\operatorname{Cay}(G,S)$ of G with respect to the \textit{connection set} $S$ is an undirected graph with $G$ as the set of vertices and two distinct vertices $u,v\in G$ are adjacent in $\operatorname{Cay}(G,S)$ if and only if $uv^{-1}\in S$. 
A set $S\subseteq G$ is said to be \textit{normal} in $G$ if $gSg^{-1}=S$ for all $g\in G$. A Cayley graph is said to be \textit{normal} if the connection set is normal.

Cayley colour digraph serves as one of the generalizations of Cayley graphs. For a finite group $G$ and a function $f:G\to\mathbb{C}$ (also called \textit{connection function}), the \textit{Cayley colour digraph} \cite{babai1979spectra}, denoted by $\operatorname{Cay}(G,f)$, is defined to be the directed graph with vertex set $G$ and arc set $\lbrace{(g,h)\mid g,h\in G}\rbrace$ such that each arc $(g,h)$ has colour $f(gh^{-1})$.
The adjacency matrix of $\operatorname{Cay}(G,f)$ is defined to be the matrix whose rows and columns are indexed by the elements of $G$, and the $(g,h)$-entry is equal to $f(gh^{-1})$, i.e., $A(\operatorname{Cay}(G,f)=[f(gh^{-1})]_{g,h\in G}$. The eigenvalues of $\operatorname{Cay}(G,f)$ are the eigenvalues of its adjacency matrix. 

Note that the adjacency matrix of $\operatorname{Cay}(G,f)$ is not symmetric in general and hence may contain both real and complex eigenvalues. In particular, when $f:G\to\lbrace{0,1}\rbrace$ and the set $S\coloneqq\lbrace{g\in G\mid f(g)=1}\rbrace$ satisfies $1\notin S$ and $S^{-1}=S$, the Cayley colour digraph $\operatorname{Cay}(G,f)$ can be identified with the Cayley graph $\operatorname{Cay}(G,S)$ and the adjacency matrix of $\operatorname{Cay}(G,f)$ coincides with that of $\operatorname{Cay}(G,S)$. Following the notion of digraph, Cayley colour graph was introduced in \cite{poddar2026algebraic}.

\begin{definition}(\cite{poddar2026algebraic})
Let $G$ be a finite group and $f:G\to\mathbb{Q}$ be such that $f(g)=f(g^{-1})$ for all $g\in G$. The Cayley colour graph, denoted by $\Gamma_f=\operatorname{Cay}(G,f)$, is defined to be the undirected graph with vertex set $G$ and edge set $\lbrace{\lbrace{g,h}\rbrace\mid f(gh^{-1})\neq 0,~g,h\in G}\rbrace$, where each edge $\lbrace{g,h}\rbrace$ has some non-zero colour $f(gh^{-1})$.
\end{definition}

A function $f:G\to\mathbb{C}$ is called a \textit{class function} if $f(ghg^{-1})=f(h)$, for all $g,h\in G$ or equivalently, if $f$ is constant on each conjugacy class of $G$. The set of all class functions, denoted by $Z(L(G))$, forms a vector space over $\mathbb{C}$, where $Z(L(G))$ is the center of the group algebra $L(G)=\mathbb{C}^G=\lbrace{f\mid f:G\to\mathbb{C}}\rbrace$.

\medskip

A graph is said to be \textit{integral} if all its adjacency eigenvalues are integers. The quest of characterizing integral graphs was first initiated by Harary and Schwenk \cite{harary1974graphs}. A great deal of research has been done on integral graphs \cite{ahmadi2009graphs,balinska2002survey}. Especially, Cayley graphs that are integral have been studied extensively over years \cite{abdollahi2009cayley,abdollahi2011integral,godsil2025integral,guo2019integral,klotz2011integral}. Alperin and Peterson \cite{alperin2012integral} have completely characterized the integrality condition for Cayley graphs over abelian groups, and later works extended these results to specific non-abelian families \cite{cheng2019integral,cheng2023integral,huang2021integral,lu2018integral}.

\medskip

The aim of this paper is to generalize the concept of Cayley integral groups, introduced by Klotz and Sander \cite{klotz2010integral}. A group $G$ is said to be \textit{Cayley integral} (or $\operatorname{CI}$\footnote{The term $\operatorname{CI}$ also stands for \textit{Cayley isomorphism} in literature, which is far from our topic. Interested readers can see \cite{li2002isomorphisms} to get a brief overview.}, in short) if $\operatorname{Cay}(G,S)$ is integral for every connection set $S$ of $G\setminus\lbrace{1}\rbrace$. Inspired by this definition, we introduce the following notions. We define the set
$$\mathfrak{F}\coloneqq\lbrace{f\in Z(L(G))\mid f(G)\subseteq\mathbb{Z}\text{ and } f(g)=f(g^{-1})\text{ for all }g\in G}\rbrace.$$

\begin{definition}\label{definition-Cay-col-int-groups}
A group $G$ is said to be
\begin{enumerate}
\item \textbf{Cayley colour integral} (or $\operatorname{\mathfrak{C}CI}$, in short) if for every function $f:G\to\mathbb{Z}$ satisfying $f(g)=f(g^{-1})$ for all $g\in G$, the Cayley colour graph $\Gamma_f=\operatorname{Cay}(G,f)$ is integral.
\item \textbf{$\mathfrak{F}$-Cayley colour integral} (or $\mathfrak{F}$-$\operatorname{\mathfrak{C}CI}$, in short) if for every function $f\in\mathfrak{F}$, the Cayley colour graph $\Gamma_f=\operatorname{Cay}(G,f)$ is integral.
\item \textbf{normal Cayley integral} (or $\operatorname{NCI}$, in short) if every normal Cayley graph $\Gamma=\operatorname{Cay}(G,S)$ on $G$ is integral.
\end{enumerate}
\end{definition}

Klotz and Sander \cite{klotz2010integral} characterized all abelian groups $G$ which are Cayley integral. Later Abdollahi, Jazarei \cite{abdollahi2014groups} and Ahmady, Bell, Mohar \cite{ahmady2014integral} independently extended their result by characterizing all non-abelian groups which are Cayley integral. For $k\in\mathbb{N}$, let
$$\mathcal{G}_k\coloneqq\lbrace{G\mid\operatorname{Cay}(G,S) \text{ is integral whenever $|S|\le k$}}\rbrace.$$
It is quite obvious that $\mathcal{G}_1$ is just the class of all finite groups, and $\mathcal{G}_2$ consists of exactly those groups whose elements are of order $1,2,3,4$ or $6$, and contain no subgroup isomorphic to $D_n$ for $n\ge 4$, where $D_n=\langle{a,b\mid a^n=b^2=1,ab=ba^{-1}}\rangle$ is the dihedral group of order $2n$. Est{\'e}lyi and Kov{\'a}cs \cite{estelyi2014groups} generalized the class of Cayley integral groups by completely classifying the groups in $\mathcal{G}_k$ for $k\ge 4$. Later, Ma and Wang \cite{ma2016finite} classified the groups in $\mathcal{G}_3$. 

\medskip

The main goal of this paper is to classify $\operatorname{\mathfrak{C}CI}$ groups, $\mathfrak{F}$-$\operatorname{\mathfrak{C}CI}$ groups and provide a few characterizations of $\operatorname{NCI}$ groups. In the process, we prove the implications as shown by Figure \ref{chain-diagram}. 




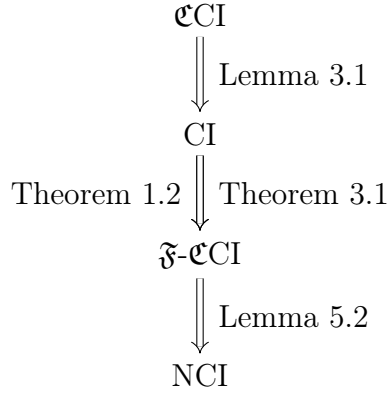
\begin{figure}[h]
\centering
\begin{tikzpicture}[
    node distance=1cm,
    imply/.style={-{Implies}, double, double distance=2pt, thin}
]

    \node (cci) {$\mathfrak{C}$CI};
    \node (ci) [below=of cci] {CI};
    \node (fcci) [below=of ci] {$\mathfrak{F}$-$\mathfrak{C}$CI};
    \node (nci) [below=of fcci] {NCI};

    \draw [imply] (cci) -- node[midway, right=0.1cm] {Lemma \ref{CCI-implies-CI-lemma}} (ci);
    \draw [imply] (ci) -- node[midway, right=0.1cm] {Theorem \ref{CI-groups-theorem}} (fcci);
    \draw [imply] (ci) -- node[midway, left=0.1cm]{Theorem \ref{finite-F-CCI-groups-theorem}} (fcci);
    \draw [imply] (fcci) -- node[midway, right=0.1cm] {Lemma \ref{F-CCI-implies-NCI-lemma}} (nci);

\end{tikzpicture}
\caption{Implications among the groups mentioned in Definition \ref{definition-Cay-col-int-groups}.}
\label{chain-diagram}
\end{figure}

\begin{remark}
The implications in Figure \ref{chain-diagram} are proper. For example,
\begin{itemize}
\item $S_3$ and $\mathbb{Z}_3\rtimes\mathbb{Z}_4$ are $\operatorname{CI}$ but not $\operatorname{\mathfrak{C}CI}$ (from Lemma \ref{S3-Dic12-non-CCI-lemma}).
\item $A_4$ is $\operatorname{\mathfrak{F}-\mathfrak{C}CI}$ but not $\operatorname{CI}$ (from Theorem \ref{CI-groups-theorem} and Theorem \ref{finite-F-CCI-groups-theorem}).
\item $Q_8\times\mathbb{Z}_3$ is $\operatorname{NCI}$ but not $\operatorname{\mathfrak{F}-\mathfrak{C}CI}$ (from Proposition \ref{NCI-direct-product-proposition} and Theorem \ref{finite-F-CCI-groups-theorem}).
\end{itemize}
\end{remark}

\subsection{Organization of the paper}

\medskip

In Section \ref{prelims-notations} we recall basic preliminaries regarding character theory and Cayley graphs. In Section \ref{cci_section} we provide the complete classification of CCI groups by proving the following theorem.

\begin{theorem}\label{CCI-groups-theorem}
A group $G$ is $\operatorname{\mathfrak{C}CI}$ if and only if $G$ is isomorphic to one of the following groups.
\begin{itemize}
\item $\mathbb{Z}_2^n\times\mathbb{Z}_3^m$, $m,n\ge 0$.
\item $\mathbb{Z}_2^n\times\mathbb{Z}_4^m$, $m,n\ge 0$.
\item $Q_8\times\mathbb{Z}_2^n$, $n\ge 0$, where $Q_8$ is the quaternion group of order $8$.
\end{itemize}
\end{theorem}

Section \ref{fcci_section} classifies the $\mathfrak{F}$-CCI groups and establishes their basic properties. In particular, we prove the following theorem.

\begin{theorem}\label{finite-F-CCI-groups-theorem}
A group $G$ is $\mathfrak{F}$-$\operatorname{\mathfrak{C}CI}$ if and only if $o(g)\in\lbrace{1,2,3,4,6}\rbrace$ for all $g\in G$.
\end{theorem}

Section \ref{nci_section} is devoted to characterizing NCI groups, investigate its subgroup and quotient properties and provide necessary condition for nilpotent $\operatorname{NCI}$ groups. Moreover, by the following theorem, we provide an alternative characterization of inverse semi-rational groups \cite{chillag2010semi}.

\begin{theorem}\label{NCI-equivalent-theorem}
Let $G$ be a group. Then the following are equivalent.
\begin{enumerate}
\item $G$ is $\operatorname{NCI}$.
\item $G$ is inverse semi-rational.
\item $\chi(g)+\chi(g^{-1})$ is an integer, for every $\chi\in\operatorname{Irr}(G)$ and $g\in G$.
\end{enumerate}
\end{theorem}



\section{Preliminaries and notations}\label{prelims-notations}

For a group $G$, we denote by $\operatorname{Irr}(G)$, the set of all inequivalent irreducible (complex) characters of $G$. A complex number $\alpha$ is said to be an \textit{algebraic integer} if $\alpha$ is a root of a monic polynomial in $\mathbb{Z}[x]$. The set of all algebraic integers, denoted by $\mathbb{A}$, forms a subring of $\mathbb{C}$. Let $\zeta_n=e^{\frac{2\pi i}{n}}$ denote the primitive $n$-th root of unity.

\begin{lemma}(\cite{isaacs1994character})\label{chi-algebraic-integer-lemma}
Let $G$ be a finite group and $\chi$ be a character of $G$. Then $\chi(g)$ is an algebraic integer for all $g\in G$. Moreover, if $|G|=n$, then $\chi(g)\in\mathbb{Q}(\zeta_n)$.
\end{lemma}
For any $S\subseteq G$ and $\chi\in\operatorname{Irr}(G)$, define
$$\chi(S)\coloneqq\sum_{s\in S}\chi(s).$$
For $g\in G$, the \textit{atom} \cite{alperin2012integral} of $g$, denoted by $\operatorname{Atom}(g)$ is defined as $\lbrace{x\in G~|~\langle{x}\rangle=\langle{g}\rangle}\rbrace$, i.e., $\operatorname{Atom}(g)=\lbrace{g^k~|~\operatorname{gcd}(k,o(g))=1}\rbrace$. A set $S\subseteq G$ is said to be \textit{Eulerian} \cite{guo2019integral} if for any $s\in S$, $\operatorname{Atom}(s)\subseteq S$. In other words, a set $S\subseteq G$ is Eulerian if $S$ is a (disjoint) union of some atoms in $G$. The following result is due to Godsil and Spiga \cite{godsil2025integral}, which provides a necessary and sufficient condition for the integrality of a normal Cayley graph.

\begin{theorem}(\cite{godsil2025integral}, Theorem $1.1$)\label{Normal-Eulerian-iff-integral-theorem}
Let $\Gamma=\operatorname{Cay}(G,S)$ be a normal Cayley graph. Then $\Gamma$ is integral if and only if $S$ is Eulerian.
\end{theorem}


An element $g\in G$ is said to be \textit{rational} or \textit{rational-valued} if $\chi(g)\in\mathbb{Q}$ for all $\chi\in\operatorname{Irr}(G)$. By Lemma \ref{chi-algebraic-integer-lemma}, this implies that $\chi(g)\in\mathbb{Z}$. A group $G$ is said to be a \textit{rational group} (or a \textit{$\mathbb{Q}$-group}) if every element of $G$ is rational, i.e., if its character table is an integer matrix. It can be shown that a group $G$ is rational if and only if for every $g\in G$, $\operatorname{Atom}(g)\subseteq C_g$ [\cite{huppert2013endliche}, Chapter V, Theorem $13.7~(b)$].

The notion of semi-rational and inverse semi-rational groups were introduced by Chillag and Dolfi \cite{chillag2010semi} as a generalization of rational groups. Let $G$ be a group and let $m$ be its exponent. An element $g\in G$ is said to be \textit{semi-rational} if there exists $r_g\in\mathbb{Z}_{m}^*$, depending on $g$, such that $\operatorname{Atom}(g)\subseteq C_g\cup C_{g^{r_g}}$. An element $g\in G$ is said to be \textit{inverse semi-rational} if $\operatorname{Atom}(g)\subseteq C_g\cup C_{g^{-1}}$. A group $G$ is said to be \textit{semi-rational} (resp. \textit{inverse semi-rational}) if every element of $G$ is semi-rational (resp. inverse semi-rational). 

Inverse semi-rational groups are also known in the literature as ``cut groups'', where cut is an acronym for ``\textbf{c}entral \textbf{u}nits are \textbf{t}rivial'', since they are precisely the finite groups for which all the central units of their integral group ring are trivial \cite{bachle2018integral,ritter1990integral}. By definition, every inverse semi-rational group is semi-rational but not conversely. For example, $D_8$, the dihedral group of order $16$ is semi-rational but not inverse semi-rational.
The authors in \cite{chillag2010semi} provided a necessary condition for a solvable$/$supersolvable group $G$ to be semi-rational$/$inverse semi-rational, in terms of the possible prime factors of $|G|$ (Theorem $2$, Proposition $6$). They also showed that a group of odd order is semi-rational if and only if it is inverse semi-rational (Remark $13$) and characterized semi-rational groups of odd order (Theorem $3$). Note that, Theorem \ref{NCI-equivalent-theorem} essentially enables us to use the terms ``$\operatorname{NCI}$'' and ``inverse semi-rational'' interchangeably.

\medskip

We prove an elementary result in this context.

\begin{lemma}\label{semi-rational-NCI-rational-center-closed-lemma}
Let $G$ be a group. Then the following results hold.
\begin{enumerate}
\item If $G$ is semi-rational, so is $Z(G)$. 
\item If $G$ is inverse semi-rational, so is $Z(G)$.
\item If $G$ is rational, so is $Z(G)$.
\end{enumerate}
\end{lemma}
\pf Let $G$ be semi-rational and let $x\in Z(G)$. Then there exists an integer $r_x$, depending on $x$, such that $\operatorname{Atom}(x)\subseteq C_x\cup C_{x^{r_x}}$. Since the conjugacy classes in $G$ and the conjugacy classes in $Z(G)$ are same, $Z(G)$ is semi-rational. This proves $(1)$. Part $(2)$ and $(3)$ follows in a similar manner.\qed

\begin{remark}\label{Q-ISR-SR-not-subgroup-closed-remark}
Subgroup of a semi-rational, inverse semi-rational and rational group may not be semi-rational, inverse semi-rational and rational, respectively. For example, $S_5$ is rational, inverse semi-rational and semi-rational but $\mathbb{Z}_5$ is none of them.
\end{remark}

At this juncture, we provide a more detailed pictorial representation of Figure \ref{chain-diagram} (Figure \ref{venn_diagram}), which will help us visualize the hierarchy better.

\begin{figure}[h]
    \centering
    \begin{tikzpicture}[scale=1.45, very thick]

\draw[thick] (-1,0) rectangle (9,6);
\draw[fill=Cyan4, opacity=0.15, thick] (-1,0) rectangle (9,6);

\def\points{}  
\foreach \a in {0,10,...,350}{
  \pgfmathsetmacro{\rx}{3.4}
  \pgfmathsetmacro{\ry}{2.4}
  \pgfmathsetmacro{\cosang}{cos(\a)}
  \pgfmathsetmacro{\sinang}{sin(\a)}
  \pgfmathsetmacro{\ellipserad}{1/sqrt( (\cosang/\rx)^2 + (\sinang/\ry)^2 )}
  \pgfmathsetmacro{\r}{\ellipserad + rand*0.3}
  \ifdim\r pt<0.2pt \pgfmathsetmacro{\r}{0.2}\fi
  \pgfmathsetmacro{\x}{4 + \r*cos(\a)}
  \pgfmathsetmacro{\y}{3 + \r*sin(\a)}
  \xdef\points{\points (\x,\y)}
}


\draw[fill=cyan, opacity=0.2, very thick] (4,3) ellipse [x radius=3.5, y radius=2.1];
\draw[thick] (4,3) ellipse [x radius=3.5, y radius=2.1];
\draw[fill=gray, opacity=0.2, very thick] (3,3) ellipse [x radius=2.1, y radius=1.5];
\draw[thick] (3,3) ellipse [x radius=2.1, y radius=1.5];
\draw[fill=violet, opacity=0.2, very thick] (3.4,3) circle (1.2);
\draw[thick] (3.4,3) circle (1.2);
\draw[fill=blue, opacity=0.2, very thick] (3.4,3) circle (0.6);
\draw[thick] (3.4,3) circle (0.6);
\draw[fill=yellow, opacity=0.2, very thick] (5.3,3) ellipse [x radius=1.8, y radius=1.1];
\draw[thick] (5.3,3) ellipse [x radius=1.8, y radius=1.1];
\draw[fill=magenta, opacity=0.15, very thick] (4,2.8) ellipse [x radius=4.6, y radius=2.65];
\draw[thick] (4,2.8) ellipse [x radius=4.6, y radius=2.65];

\node[] at (5.9,3) {\bf \textcolor{blue}{$\mathbb{Q}$-group}};
\node[] at (3.15,3) {\bf \textcolor{blue}{CCI}};
\node[] at (2.5,3) {\bf \textcolor{blue}{CI}};
\node[] at (5.5,4.5) {\bf \textcolor{blue}{NCI}};
\node[] at (1.5,3) {\bf \textcolor{blue}{$\mathfrak{F}$-CCI}};
\node[] at (4,0.6) {\bf \textcolor{blue}{Semi-rational}};

\node[] at (4.29,3) { \textcolor{black}{$S_3$}};
\node[] at (4.85,3) { \textcolor{black}{$D_4$}};
\node[] at (7.8,5) { \textcolor{black}{$D_7$}};
\node[] at (8,3) { \textcolor{black}{$D_8$}};
\node[] at (5.8,3.6) { \textcolor{black}{$S_5$}};
\node[] at (3.8,3) { \textcolor{black}{$Q_8$}};
\node[] at (3.3,3.8) { \textcolor{black}{$\mathbb{Z}_3\rtimes \mathbb{Z}_4$}};
\node[] at (3.4,4.7) { \textcolor{black}{$Q_8\times\mathbb{Z}_3$}};
\node[] at (3.3,3.3) { \textcolor{black}{$\mathbb{Z}_3$}};
\node[] at (1.7,3.5) { \textcolor{black}{$A_4$}};
\end{tikzpicture}
    \caption{Pictorial representation of the hierarchy between the groups.}
    \label{venn_diagram}
\end{figure}
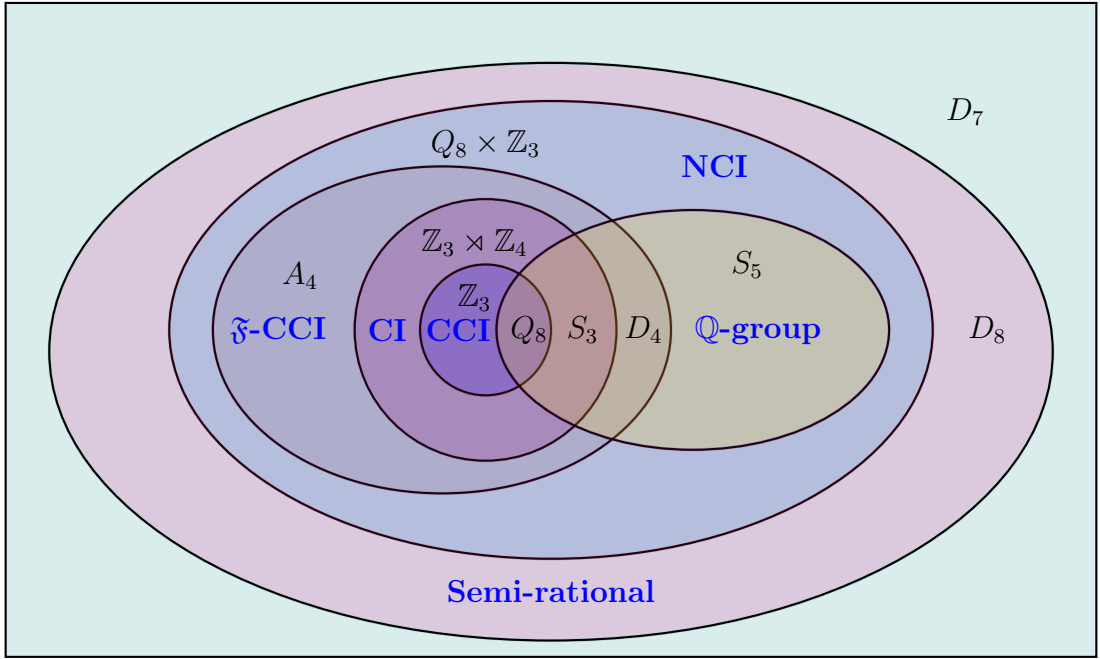

\medskip 

If $f:G\to\mathbb{C}$ is a class function, the spectrum of the Cayley colour graph $\Gamma_f$ can be obtained as following.

\begin{lemma}(\cite{foster2016spectra}, Theorem $4.3$)\label{spectrum-Cay-col-digraph}
Let $f:G\to\mathbb{C}$ be a class function. Then the spectrum of the Cayley colour digraph $\operatorname{Cay}(G,f)$ is given by
$$\lambda_\chi=\frac{1}{\chi(1)}\sum_{g\in G}f(g)\chi(g),$$
where $\chi\in\operatorname{Irr}(G)$ with multiplicity of $\lambda_\chi$ being $\chi(1)^2$.
\end{lemma}

Let $f:G\to\mathbb{C}$ be a function. For $k\in\mathbb{Z}$, we define $f^k:G\to\mathbb{C}$ as $f^k(g)=f(g^k)$ for all $g\in G$. We recall the following result from \cite{poddar2026algebraic} which will be in use later.

\begin{lemma}(\cite{poddar2026algebraic}, Corollary $3.6$)\label{Cay-col-graph-integral-iff-cond-corollary}
Let $G$ be a group of order $n$ and let $f:G\to\mathbb{Z}$ be a class function. Then the Cayley colour graph $\Gamma_f=\operatorname{Cay}(G,f)$ is integral if and only if $f^h=f$ for all $h\in\mathbb{Z}_n^{*}$. 
\end{lemma}

It is known from Galois theory that, $\operatorname{Gal}(\mathbb{Q}(\zeta_n)/\mathbb{Q})\cong\mathbb{Z}_n^{*}$. Let $\eta:\operatorname{Gal}(\mathbb{Q}(\zeta_n)/\mathbb{Q})\to\mathbb{Z}_n^{*}$ be the isomorphism 
defined as $\eta(\sigma)=h$ for any $\sigma\in\operatorname{Gal}(\mathbb{Q}(\zeta_n)/\mathbb{Q})$, where $h\in\mathbb{Z}_n^{*}$ is the integer such that $\sigma(\zeta_n)=\zeta_n^{h}$.

\begin{lemma}(\cite{huppert2013endliche}, Chapter V, Theorem $13.1~(c)$)\label{sigma-eta-chi-lemma}
Let $G$ be a group of order $n$ and $\sigma\in\operatorname{Gal}(\mathbb{Q}(\zeta_n)/\mathbb{Q})$. Then for any character $\chi$ of $G$, $\sigma(\chi(g))=\chi(g^{\eta(\sigma)})$ for all $g\in G$.
\end{lemma}

\section{Cayley colour integral (CCI) groups}\label{cci_section}

In \cite{klotz2010integral}, Klotz and Sander classified the abelian Cayley integral groups. They showed that the abelian Cayley integral groups are precisely the  abelian groups of exponent dividing $4$
or $6$. In [\cite{klotz2010integral}, p.$12$, Problem $3$], they proposed the problem of characterizing all
non-abelian Cayley integral groups. Abdollahi and Jazaeri \cite{abdollahi2014groups}, and Ahmady, Bell and Mohar \cite{ahmady2014integral} independently completed the classification of such groups. They proved the following remarkable result.

\begin{theorem}\label{CI-groups-theorem}(\cite{abdollahi2014groups,ahmady2014integral})
Let $G$ be a Cayley integral group. Then $G$ is isomorphic to one of the following groups.
\begin{itemize}
\item $\mathbb{Z}_2^n\times\mathbb{Z}_3^m$, $m,n\ge 0$.
\item $\mathbb{Z}_2^n\times\mathbb{Z}_4^m$, $m,n\ge 0$.
\item The symmetric group $S_3$.
\item The dicyclic group $\operatorname{Dic}_{12}$ isomorphic to $\mathbb{Z}_3\rtimes\mathbb{Z}_4=\lbrace{a,b\mid a^6=1,a^3=b^2, bab^{-1}=a^{-1}}\rbrace$.
\item $Q_8\times\mathbb{Z}_2^n$, $n\ge 0$, where $Q_8$ is the quaternion group of order $8$.
\end{itemize}
\end{theorem}

In this section, we prove Theorem \ref{CCI-groups-theorem} and thereby completely classify all Cayley colour integral groups. We begin with an elementary lemma.

\begin{lemma}\label{CCI-implies-CI-lemma}
Every $\operatorname{\mathfrak{C}CI}$ group is $\operatorname{CI}$.
\end{lemma}
\pf Let $G$ be a $\operatorname{\mathfrak{C}CI}$ group and let $S\subseteq G\setminus\lbrace{1}\rbrace$ be inverse-closed. Define the characteristic function $\delta_S:G\to\lbrace{0,1}\rbrace$ on $S$ as
$$\delta_S(x)=\begin{cases}
    1, & \text{if $x\in S$},\\
    0, & \text{if $x\notin S$.}
\end{cases}$$
Since $S^{-1}=S$, $\delta_S(g)=\delta_S(g^{-1})$ for all $g\in G$. Thus the adjacency matrix of the Cayley graph $\Gamma=\operatorname{Cay}(G,S)$ can be identified with that of the Cayley colour graph $\Gamma_{\delta_S}=\operatorname{Cay}(G,\delta_S)$. Since $G$ is $\operatorname{\mathfrak{C}CI}$, $\Gamma_{\delta_S}$ is integral and hence $\Gamma$ is integral. Hence $G$ is a $\operatorname{CI}$ group.\qed

\begin{lemma}\label{S3-Dic12-non-CCI-lemma}
$S_3$ and $\mathbb{Z}_3\rtimes\mathbb{Z}_4$ are not $\operatorname{\mathfrak{C}CI}$ groups.
\end{lemma}
\pf $1).$ For the group $S_3=\langle{a,b\mid a^3=b^2=1,ab=ba^{-1}}\rangle$, consider the function $\alpha:S_3\to\mathbb{Z}$ defined by
$$\alpha(1)=0,\alpha(a)=\alpha(a^2)=1,\alpha(b)=3,\alpha(ba)=7,\alpha(ba^2)=4.$$
It can be shown using SageMath \cite{stein2007sage} that the spectrum of $\Gamma_\alpha=\operatorname{Cay}(S_3,\alpha)$ is given by
$$\operatorname{Spec}(\Gamma_\alpha)=\begin{pmatrix}
    16 & 12 & -1+\sqrt{13} & -1-\sqrt{13}\\
    1 & 1 & 2 & 2
\end{pmatrix}.$$

\medskip

$2).$ For the group $\mathbb{Z}_3\rtimes\mathbb{Z}_4=\lbrace{a,b\mid a^6=1,a^3=b^2, bab^{-1}=a^{-1}}\rbrace$, consider the function $\beta:\mathbb{Z}_3\rtimes\mathbb{Z}_4\to\mathbb{Z}$ defined by
$$\beta(1)=0,\beta(a)=\beta(a^5)=1,\beta(a^2)=\beta(a^4)=7,$$
$$\beta(b)=\beta(b^3)=3,\beta(b^2)=8,\beta(ab)=\beta(ab^3)=4,\beta(a^5b)=\beta(a^5b^3)=5.$$
Then using SageMath the spectrum of $\Gamma_\beta=\operatorname{Cay}(\mathbb{Z}_3\rtimes\mathbb{Z}_4,\beta)$ can be obtained as
$$\operatorname{Spec}(\Gamma_\beta)=\begin{pmatrix}
    48 & 4 & 0 & -14 & 2\sqrt{3} & -2\sqrt{3}\\
    1 & 2 & 1 & 4 & 2 & 2
\end{pmatrix}.$$
Hence $S_3$ and $\mathbb{Z}_3\rtimes\mathbb{Z}_4$ are not $\operatorname{\mathfrak{C}CI}$.\qed\\
\\
\medskip
\textbf{Proof of Theorem \ref{CCI-groups-theorem}:}

\medskip

Note that any function $f:G\to\mathbb{Z}$ on an abelian group $G$ is a class function. Therefore, all the abelian $\operatorname{\mathfrak{C}CI}$ groups are precisely the abelian $\mathfrak{F}$-$\operatorname{\mathfrak{C}CI}$ groups. From Theorem \ref{finite-F-CCI-groups-theorem}, we obtain that the abelian $\mathfrak{F}$-$\operatorname{\mathfrak{C}CI}$ groups are $\mathbb{Z}_2^n\times\mathbb{Z}_3^m$ and $\mathbb{Z}_2^n\times\mathbb{Z}_4^m$, where $m,n\ge 0$. Also note that the conjugacy classes of the quaternion group $Q_8=\lbrace{1,-1,i,-i,j,-j,k,-k}\rbrace$ are $\lbrace{1}\rbrace,\lbrace{-1}\rbrace,\lbrace{i,-i}\rbrace,\lbrace{j,-j}\rbrace,\lbrace{k,-k}\rbrace$. Hence the conjugacy classes of $Q_8\times\mathbb{Z}_2^n$ become $$\lbrace{(1,x_1,\ldots,x_n)}\rbrace, \lbrace{(-1,x_1,\ldots,x_n)}\rbrace, \lbrace{(i,x_1,\ldots,x_n),(-i,x_1,\ldots,x_n)}\rbrace,$$
$$\lbrace{(j,x_1,\ldots,x_n),(-j,x_1,\ldots,x_n)}\rbrace,\lbrace{(k,x_1,\ldots,x_n),(-k,x_1,\ldots,x_n)}\rbrace.$$
where $x_i\in\mathbb{Z}_2$ for all $i\in\lbrace{1,\ldots,n}\rbrace$.

Since each conjugacy class of $Q_8\times\mathbb{Z}_2^n$ is inverse-closed, we have that every $f:Q_8\times\mathbb{Z}_2^n\to\mathbb{Z}$ satisfying $f(x)=f(x^{-1})$ for all $x\in Q_8\times\mathbb{Z}_2^n$ is essentially a class function. Hence by Theorem \ref{finite-F-CCI-groups-theorem}, the Cayley colour graph $\operatorname{Cay}(Q_8\times\mathbb{Z}_2^n,f)$ is integral for each such $f$. Combining all these, together with Lemma \ref{CCI-implies-CI-lemma}, \ref{S3-Dic12-non-CCI-lemma} and Theorem \ref{CI-groups-theorem}, we have the result.\qed

\begin{corollary}
Subgroup and quotient of a $\operatorname{\mathfrak{C}CI}$ group is $\operatorname{\mathfrak{C}CI}$.
\end{corollary}

\section{$\mathfrak{F}$-Cayley colour integral ($\mathfrak{F}$-$\mathfrak{C}$CI) groups}\label{fcci_section}

We have already seen that subgroup of a $\operatorname{\mathfrak{C}CI}$ group is $\operatorname{\mathfrak{C}CI}$. Also from [\cite{abdollahi2014groups}, Lemma $2.2$], we know that subgroup of a $\operatorname{CI}$ group is $\operatorname{CI}$. We now show that this property also holds for $\mathfrak{F}$-$\operatorname{\mathfrak{C}CI}$ groups. 


\begin{lemma}\label{subgroup-of-F-CCI-is-CCI-lemma}
Subgroup of a $\mathfrak{F}$-$\operatorname{\mathfrak{C}CI}$ group is $\mathfrak{F}$-$\operatorname{\mathfrak{C}CI}$.
\end{lemma}
\pf Let $G$ be a $\mathfrak{F}$-$\operatorname{\mathfrak{C}CI}$ group and $H$ be an arbitrary subgroup of $G$. Let $\psi:H\to\mathbb{Z}$ be a class function on $H$ satisfying $\psi(h)=\psi(h^{-1})$ for all $h\in H$. Define $\tilde{\psi}:G\to\mathbb{Z}$ as
$$\tilde{\psi}(x)=\begin{cases}
    \psi(x),& \text{ if $x\in H$},\\
    0,& \text{ if $x\notin H$}.
\end{cases}$$
Now define $\Psi:G\to\mathbb{Z}$ by 
$$\Psi(g)=\sum_{x\in G}\tilde{\psi}(xgx^{-1}).$$
Clearly $\Psi(g)=\Psi(g^{-1})$ for all $g\in G$. Also it follows that $\Psi$ is a class function on $G$. Thus $\operatorname{Res}_{H}^{G}\Psi$ is a class function on $H$. Since $G$ is $\mathfrak{F}$-$\operatorname{\mathfrak{C}CI}$, the Cayley colour graph $\Gamma_{\Psi}=\operatorname{Cay}(G,\Psi)$ is integral. Hence by Lemma \ref{spectrum-Cay-col-digraph}, $\sum_{g\in G}\Psi(g)\chi(g)$ is an integer for all $\chi\in\operatorname{Irr}(G)$. Let $\mu\in\operatorname{Irr}(H)$ be arbitrary. Then the induced character $\operatorname{Ind}_{H}^{G}\mu$ on $G$ can be expressed as $\operatorname{Ind}_{H}^{G}\mu=\sum_{\chi\in\operatorname{Irr}(G)}c_\chi\chi$, where $c_\chi=\left\langle{\operatorname{Ind}_{H}^{G}\mu,\chi}\right\rangle$ denotes the multiplicity of $\chi$ in $\operatorname{Ind}_{H}^{G}\mu$. Now we have,
$$\sum_{h\in H}\Psi(h)\mu(h)=\sum_{h\in H}\left(\sum_{x\in H}\tilde{\psi}(xhx^{-1})\right)\mu(h)=\sum_{x\in H}\left(\sum_{h\in H}\tilde{\psi}(xhx^{-1})\mu(h)\right).$$
Rewriting $xhx^{-1}=t$ and re-indexing the sum, we have
$$\sum_{h\in H}\Psi(h)\mu(h)=\sum_{x\in H}\left(\sum_{t\in H}\tilde{\psi}(t)\mu(x^{-1}tx)\right)=\sum_{x\in H}\left(\sum_{t\in H}\psi(t)\mu(t)\right) =|H|\left(\sum_{h\in H}\psi(h)\mu(h)\right).$$
Using Fr{\"o}benius reciprocity, we obtain
\begin{align*}
\sum_{h\in H}\psi(h)\mu(h)&=\left\langle{\mu,\operatorname{Res}_{H}^{G}\Psi}\right\rangle_{H}\\
&=\left\langle{\operatorname{Ind}_{H}^{G}\mu,\Psi}\right\rangle_{G}\\
&=\left\langle{\sum_{\chi\in\operatorname{Irr}(G)}c_\chi\chi,\Psi}\right\rangle_{G}\\
&=\sum_{\chi\in\operatorname{Irr}(G)}c_\chi\left\langle{\chi,\Psi}\right\rangle_{G}\\
&=\sum_{\chi\in\operatorname{Irr}(G)}c_\chi\left(\frac{1}{|G|}\sum_{g\in G}\chi(g)\Psi(g)\right)\\
&=\frac{1}{|G|}\sum_{\chi\in\operatorname{Irr}(G)}c_\chi\left(\sum_{g\in G}\chi(g)\Psi(g)\right).
\end{align*}
Hence $\frac{1}{\mu(1)}\sum_{h\in H}\psi(h)\mu(h)\in\mathbb{Q}$. By Lemma \ref{spectrum-Cay-col-digraph}, $\frac{1}{\mu(1)}\sum_{h\in H}\psi(h)\mu(h)\in\operatorname{Spec}(\operatorname{Cay}(H,\psi))$. Since eigenvalues are algebraic integers, we have $\frac{1}{\mu(1)}\sum_{h\in H}\psi(h)\mu(h)\in\mathbb{Z}$. Thus $\operatorname{Cay}(H,\psi)$ is integral. By the arbitrariness of $\psi$, we have $H$ is an $\mathfrak{F}$-$\operatorname{\mathfrak{C}CI}$ group. This completes the proof.\qed\\
\\
\medskip
\textbf{Proof of Theorem \ref{finite-F-CCI-groups-theorem}:}

\medskip

We first prove the sufficiency. Let $G$ be a group of order $n$ such that $o(g)\in\lbrace{1,2,3,4,6}\rbrace$ for all $g\in G$. Let $f\in\mathfrak{F}$ and consider the Cayley colour graph $\Gamma_f=\operatorname{Cay}(G,f)$. Fix any $h\in\mathbb{Z}_n^{*}$. Obviously, $f^h(1)=f(1)$. We now consider four cases.

\medskip

\textbf{Case 1:} Let $g\in G$ be such that $o(g)=2$. Then $h\equiv 1\mod{2}$ and hence $f^h(g)=f(g^h)=f(g)$.

\medskip

\textbf{Case 2:} Let $g\in G$ be such that $o(g)=3$. Then $h\equiv 1$ or $2\mod{3}$, i.e., $h=3t+1$ or $3t+2$ for some integer $t\ge 0$. For $h=3t+1$, we have $f^h(g)=f(g^h)=f(g^{3t+1})=f(g)$ and for $h=3t+2$, we have $f^h(g)=f(g^h)=f(g^{3t+2})=f(g^2)=f(g^{-1})=f(g)$.

\medskip

\textbf{Case 3:} Let $g\in G$ be such that $o(g)=4$. Then $h\equiv 1$ or $3\mod{4}$, i.e., $h=4t+1$ or $4t+3$ for some integer $t\ge 0$. For $h=4t+1$, we have $f^h(g)=f(g^h)=f(g^{4t+1})=f(g)$ and for $h=4t+3$, we have $f^h(g)=f(g^h)=f(g^{4t+3})=f(g^3)=f(g^{-1})=f(g)$.

\medskip

\textbf{Case 4:} Let $g\in G$ be such that $o(g)=6$. Then $h\equiv 1$ or $5\mod{6}$, i.e., $h=6t+1$ or $6t+5$ for some integer $t\ge 0$. For $h=6t+1$, we have $f^h(g)=f(g^h)=f(g^{6t+1})=f(g)$ and for $h=6t+5$, we have $f^h(g)=f(g^h)=f(g^{6t+5})=f(g^5)=f(g^{-1})=f(g)$.

Combining all the cases, we have $f^h(g)=f(g)$ for all $g\in G$. Since $h\in\mathbb{Z}_n^{*}$ is arbitrary, we have $f^h=f$ for all $h\in\mathbb{Z}_n^{*}$. Hence by Lemma \ref{Cay-col-graph-integral-iff-cond-corollary}, $\Gamma_f$ is integral. By the arbitrariness of $f$, we have the result. 

\medskip

To prove the necessity, let $G$ be a $\mathfrak{F}$-$\operatorname{\mathfrak{C}CI}$ group. Suppose there exists $x\in G\setminus\lbrace{1}\rbrace$ such that $o(x)\notin\lbrace{2,3,4,6}\rbrace$. Let $S=\lbrace{x,x^{-1}}\rbrace$. Then $\langle{S}\rangle=\langle{x}\rangle$. Consider $\Gamma=\operatorname{Cay}(\langle{S}\rangle,S)$. Then $S$ is normal in $\langle{S}\rangle$ and hence $\delta_S$ is a class function on $\langle{S}\rangle$. Then the adjacency matrix of the Cayley colour graph $\Gamma_{\delta_S}=\operatorname{Cay}(\langle{S}\rangle,\delta_S)$ is identified with that of the Cayley graph $\Gamma$. By Lemma \ref{subgroup-of-F-CCI-is-CCI-lemma}, $\langle{S}\rangle$ is $\mathfrak{F}$-$\operatorname{\mathfrak{C}CI}$ and hence $\Gamma_{\delta_S}$, i.e., $\Gamma$ is integral. But since $|S|=2$, $\operatorname{Cay}(\langle{S}\rangle,S)$ is a connected $2$-regular graph which is isomorphic to the cycle $\mathcal{C}_{o(x)}$, which is integral only for $o(x)=3,4,6$; a contradiction.\qed

\medskip

From Theorem \ref{CI-groups-theorem} and \ref{finite-F-CCI-groups-theorem}, we have the following immediate facts.

\begin{corollary}\label{CI-implies-F-CCI-corollary}
Every $\operatorname{CI}$ group is $\mathfrak{F}$-$\operatorname{\mathfrak{C}CI}$.
\end{corollary}

\begin{corollary}
Quotient of a $\operatorname{\mathfrak{F}-\mathfrak{C}CI}$ group is $\operatorname{\mathfrak{F}-\mathfrak{C}CI}$.
\end{corollary}

As the order of a $\mathfrak{F}$-$\operatorname{\mathfrak{C}CI}$ group is of the form $2^a3^b$, a $\mathfrak{F}$-$\operatorname{\mathfrak{C}CI}$ group is always solvable. However, we can say something more if we impose an additional restriction.
\begin{proposition}
Let $G$ be $\mathfrak{F}$-$\operatorname{\mathfrak{C}CI}$ such that $12\nmid |G|$. Then $G$ is supersolvable. 
\end{proposition}
\pf By Theorem \ref{finite-F-CCI-groups-theorem}, it is clear that $|G|=2^a3^b$. As $12\nmid |G|$, either $a\le 1$ or $b=0$. If $a=0$ or $b=0$, $G$ is nilpotent and hence supersolvable. So we assume $a=1$ and $b\ge 1$, i.e., $|G|=2\cdot3^b$. As $3\equiv 1\mod 2$, by [\cite{conrad2020subgroup}, Theorem $4.26$], $G$ is supersolvable.\qed 

\section{Normal Cayley integral (NCI) groups}\label{nci_section}

In this section, we provide a few characterization of normal Cayley integral groups and study its properties. The spectrum of a normal Cayley graph can be obtained from Lemma \ref{spectrum-Cay-col-digraph} by replacing $f$ with the class function $\delta_S$. 

\begin{lemma}\label{normal-Cay-spectrum-lemma}(\cite{foster2016spectra},\cite{zieschang1988cayley})
The spectrum of the normal Cayley graph $\Gamma=\operatorname{Cay}(G,S)$ is given by
$$\lambda_\chi=\frac{1}{\chi(1)}\chi(S)$$
where $\chi\in\operatorname{Irr}(G)$ with multiplicity of $\lambda_\chi$ being $\chi(1)^2$.
\end{lemma}

\begin{lemma}\label{F-CCI-implies-NCI-lemma}
Every $\mathfrak{F}$-$\operatorname{\mathfrak{C}CI}$ group is $\operatorname{NCI}$.
\end{lemma}
\pf Let $G$ be an $\mathfrak{F}$-$\operatorname{\mathfrak{C}CI}$ group and let $S\subseteq G\setminus\lbrace{1}\rbrace$ be a normal, inverse-closed subset of $G$. Then the characteristic function $\delta_S:G\to\lbrace{0,1}\rbrace$ is a class function on $G$. Since the adjacency matrix of the normal Cayley graph $\Gamma=\operatorname{Cay}(G,S)$ is identified with that of the Cayley colour graph $\Gamma_{\delta_S}=\operatorname{Cay}(G,\delta_S)$, we have $\Gamma_{\delta_S}$ is integral and hence $\Gamma$ is integral. Hence $G$ is an $\operatorname{NCI}$ group.\qed\\
\\
\medskip
\textbf{Proof of Theorem \ref{NCI-equivalent-theorem}:} 

\medskip

$\mathbf{(1)\Rightarrow (2)}$: Let $G$ be $\operatorname{NCI}$. Let $1\neq g\in G$ and consider $C_g$, the conjugacy class of $g$ in $G$. Take $S\coloneqq C_g\cup C_{g^{-1}}$. Clearly, $S$ is normal and inverse-closed. Thus by hypothesis, the normal Cayley graph $\Gamma=\operatorname{Cay}(G,S)$ is integral. By Theorem \ref{Normal-Eulerian-iff-integral-theorem}, $S$ is Eulerian. As $g\in S$, $\operatorname{Atom}(g)\subseteq S$. Thus $\operatorname{Atom}(g)\subseteq C_g\cup C_{g^{-1}}$. Hence $G$ is inverse semi-rational.

\medskip

$\mathbf{(2)\Rightarrow (3)}$: Let $G$ be inverse semi-rational. Consider the isomorphism $\eta:\operatorname{Gal}(\mathbb{Q}(\zeta_n)/\mathbb{Q})\to\mathbb{Z}_n^{*}$, where $n=|G|$. Let $\sigma\in\operatorname{Gal}(\mathbb{Q}(\zeta_n)/\mathbb{Q})$. Then $\eta(\sigma)\in\mathbb{Z}_n^{*}$. Let $\chi\in\operatorname{Irr}(G)$ and $g\in G$. Then by Lemma \ref{sigma-eta-chi-lemma},
$$\sigma(\chi(g)+\chi(g^{-1}))=\sigma(\chi(g))+\sigma(\chi(g^{-1}))=\chi(g^{\eta(\sigma)})+\chi(g^{-\eta(\sigma)}).$$
If $g^{\eta(\sigma)}$ is conjugate to $g$, then
$$\chi(g^{\eta(\sigma)})+\chi(g^{-\eta(\sigma)})=\chi(g)+\chi(g^{-1}).$$
If $g^{\eta(\sigma)}$ is conjugate to $g^{-1}$, then
$$\chi(g^{\eta(\sigma)})+\chi(g^{-\eta(\sigma)})=\chi(g^{-1})+\chi(g).$$
In any case we have, $\sigma(\chi(g)+\chi(g^{-1}))=\chi(g)+\chi(g^{-1})$. Thus $\chi(g)+\chi(g^{-1})\in\mathbb{Q}$. By Lemma \ref{chi-algebraic-integer-lemma}, it follows that $\chi(g)+\chi(g^{-1})$ is an integer.

\medskip

$\mathbf{(3)\Rightarrow (1)}$: Let $G$ be a group such that $\chi(g)+\chi(g^{-1})\in\mathbb{Z}$ for every $\chi\in\operatorname{Irr}(G)$ and $g\in G$. Let $S\subseteq G\setminus\lbrace{1}\rbrace$ be a normal, inverse-closed subset of $G$ and consider the Cayley graph $\Gamma=\operatorname{Cay}(G,S)$. By hypothesis, $\chi(S)\in\mathbb{Z}$ for every $\chi\in\operatorname{Irr}(G)$. Thus from Lemma \ref{normal-Cay-spectrum-lemma}, $\lambda_\chi\in\mathbb{Q}$ for all $\chi\in\operatorname{Irr}(G)$. Since eigenvalues are algebraic integers, we have $\lambda_\chi\in\mathbb{Z}$ for all $\chi\in\operatorname{Irr}(G)$. Thus the normal Cayley graph $\Gamma$ is integral. Hence $G$ is an $\operatorname{NCI}$ group.\qed

\begin{corollary}\label{Rational-group-NCI-corollary}
Every $\mathbb{Q}$-group is $\operatorname{NCI}$. In particular, every symmetric group $S_n$ is $\operatorname{NCI}$ (\cite{chen2014eigenvalues}, Corollary $1.2$).
\end{corollary}

\begin{corollary}\label{abelian-CCI-CI-FCCI-NCI-equivalence-corollary}
The following are equivalent for an abelian group $G$.
\begin{enumerate}
\item $G$ is $\operatorname{\mathfrak{C}CI}$.
\item $G$ is $\operatorname{CI}$.
\item $G$ is $\mathfrak{F}$-$\operatorname{\mathfrak{C}CI}$.
\item $G$ is $\operatorname{NCI}$.
\item $G$ is semi-rational.
\item $G\cong\mathbb{Z}_2^n\times\mathbb{Z}_3^m$ or $G\cong\mathbb{Z}_2^n\times\mathbb{Z}_4^m$, for $m,n\ge 0$.
\end{enumerate}
\end{corollary}
\pf $(1)$ implies $(2)$, $(2)$ implies $(3)$, $(3)$ implies $(4)$ follow directly from Figure \ref{chain-diagram}. $(4)$ implies $(5)$ follows from Theorem \ref{NCI-equivalent-theorem}. $(5)$ implies $(6)$ follows from [\cite{chillag2010semi}, Lemma $5$, $(4)$]. And $(6)$ implies $(1)$ follows from Theorem \ref{CCI-groups-theorem}.\qed

\medskip


The following result serves as an alternative characterization of $\operatorname{NCI}$ groups, in terms of Cayley colour graphs.

\begin{proposition}\label{NCI-iff-condition-Cay-col-proposition}
$G$ is an $\operatorname{NCI}$ group if and only if the Cayley colour graph $\Gamma_{\chi+\overline{\chi}}=\operatorname{Cay}(G,\chi+\overline{\chi})$ is integral for every $\chi\in\operatorname{Irr}(G)$, where $\overline{~\cdot~}$ denotes the complex conjugation.
\end{proposition}
\pf Let $\chi\in\operatorname{Irr}(G)$ and $g\in G$. Since $\Gamma_{\chi+\overline{\chi}}=\operatorname{Cay}(G,\chi+\overline{\chi})$ is integral, by Lemma \ref{Cay-col-graph-integral-iff-cond-corollary}, $(\chi+\overline{\chi})^h=\chi+\overline{\chi}$ for all $h\in\mathbb{Z}_n^{*}$, where $n=|G|$. Consider the isomorphism $\eta:\operatorname{Gal}(\mathbb{Q}(\zeta_n)/\mathbb{Q})\to\mathbb{Z}_n^{*}$. Let $\sigma\in\operatorname{Gal}(\mathbb{Q}(\zeta_n)/\mathbb{Q})$. Then $\eta(\sigma)\in\mathbb{Z}_n^{*}$. Then by Lemma \ref{sigma-eta-chi-lemma}, we have
\begin{align*}
\sigma(\chi(g)+\chi(g^{-1}))&=\chi(g^{\eta(\sigma)})+\overline{\chi}(g^{\eta(\sigma)})\\
&=(\chi+\overline{\chi})^{\eta(\sigma)}(g)\\
&=(\chi+\overline{\chi})(g)\\
&=\chi(g)+\chi(g^{-1}).
\end{align*}

Since $\sigma$ is arbitrary, we have $\chi(g)+\chi(g^{-1})\in\mathbb{Q}$ which implies $\chi(g)+\chi(g^{-1})\in\mathbb{Z}$ by Lemma \ref{chi-algebraic-integer-lemma}. By Theorem \ref{NCI-equivalent-theorem}, $G$ is $\operatorname{NCI}$.

Conversely, let $G$ be an $\operatorname{NCI}$ group. Let $\chi\in\operatorname{Irr}(G)$ and $g\in G$. By Theorem \ref{NCI-equivalent-theorem}, $\chi(g)+\chi(g^{-1})\in\mathbb{Z}$. Thus $\sigma(\chi(g)+\chi(g^{-1}))=\chi(g)+\chi(g^{-1})$ for all $\sigma\in\operatorname{Gal}(\mathbb{Q}(\zeta_n)/\mathbb{Q})$. Let $h'\in\mathbb{Z}_n^{*}$. Then there exists $\sigma'\in\operatorname{Gal}(\mathbb{Q}(\zeta_n)/\mathbb{Q})$, such that $\eta(\sigma')=h'$. Thus by Lemma \ref{sigma-eta-chi-lemma}, we have
\begin{align*}
(\chi+\overline{\chi})^{h'}(g)&=\chi(g^{h'})+\chi((g^{-1})^{h'})\\
&=\chi(g^{\eta(\sigma')})+\chi((g^{-1})^{\eta(\sigma')})\\
&=\sigma'(\chi(g)+\chi(g^{-1}))\\
&=\chi(g)+\chi(g^{-1})\\
&=(\chi+\overline{\chi})(g).
\end{align*}

Hence $(\chi+\overline{\chi})^{h'}=(\chi+\overline{\chi})$. Since $h'\in\mathbb{Z}_n^{*}$ is arbitrary, by Lemma \ref{Cay-col-graph-integral-iff-cond-corollary}, the Cayley colour graph $\Gamma_{\chi+\overline{\chi}}=\operatorname{Cay}(G,\chi+\overline{\chi})$ is integral.\qed

\medskip

We have already seen from Remark \ref{Q-ISR-SR-not-subgroup-closed-remark} that subgroup of an $\operatorname{NCI}$ group may not be $\operatorname{NCI}$. However, from [\cite{chillag2010semi}, Lemma $4$], one can see that quotients preserves the $\operatorname{NCI}$ property. We give an alternative proof of that result below.

\begin{proposition}\label{NCI-quotient-closed-proposition}(\cite{chillag2010semi}, Lemma $4$)\label{quotient-of-NCI-group-NCI}
Quotient of an $\operatorname{NCI}$ group is $\operatorname{NCI}$.
\end{proposition}
\pf Let $G$ be an $\operatorname{NCI}$ group and $H$ be a normal subgroup of $G$. Let $\chi\in\operatorname{Irr}(G/H)$ and $gH\in G/H$. Then $\chi(gH)+\chi(g^{-1}H)=\hat{\chi}(g)+\hat{\chi}(g^{-1})$ where $\hat{\chi}=\chi\pi\in\operatorname{Irr}(G)$. By Theorem \ref{NCI-equivalent-theorem}, $\chi(gH)+\chi(g^{-1}H)\in\mathbb{Z}$ and hence $G/H$ is $\operatorname{NCI}$.\qed

\medskip

For two $\operatorname{NCI}$ groups $G$ and $H$, the direct product $G\times H$ may not be $\operatorname{NCI}$. For example, both $\mathbb{Z}_3$ and $\mathbb{Z}_4$ are $\operatorname{NCI}$ groups while $\mathbb{Z}_3\times\mathbb{Z}_4\cong\mathbb{Z}_{12}$ is not. However, we have the following result.

\begin{proposition}\label{NCI-direct-product-proposition}
Let $G$ be a $\mathbb{Q}$-group and $H$ be an $\operatorname{NCI}$ group. Then $G\times H$ is an $\operatorname{NCI}$ group. Conversely, if $G\times H$ is $\operatorname{NCI}$, then both $G$ and $H$ are $\operatorname{NCI}$.
\end{proposition}
\pf Let $\Phi$ be an irreducible character of $G\times H$. Then $\Phi=\chi\varphi$ for some $\chi\in\operatorname{Irr}(G)$ and $\varphi\in\operatorname{Irr}(H)$. Let $(g,h)\in G\times H$. Now 
$$\Phi(g,h)+\Phi((g,h)^{-1})=\chi(g)\varphi(h)+\chi(g^{-1})\varphi(h^{-1}).$$
Since $G$ is rational, $g$ is conjugate to $g^{-1}$. Thus $\chi(g)=\chi(g^{-1})\in\mathbb{Z}$. Hence 
$$\Phi(g,h)+\Phi((g,h)^{-1})=\chi(g)(\varphi(h)+\varphi(h^{-1})).$$
Since $H$ is $\operatorname{NCI}$, by Theorem \ref{NCI-equivalent-theorem}, $\varphi(h)+\varphi(h^{-1})\in\mathbb{Z}$. Thus we have $\Phi(g,h)+\Phi((g,h)^{-1})\in\mathbb{Z}$. Hence $G\times H$ is $\operatorname{NCI}$ by Theorem \ref{NCI-equivalent-theorem}. The converse part follows from Proposition \ref{NCI-quotient-closed-proposition}.\qed




\begin{proposition}
Let $G$ be a nilpotent $\operatorname{NCI}$ group. Then $G$ is a $\lbrace{2,3}\rbrace$-group. Moreover, the Sylow subgroups of $G$ are $\operatorname{NCI}$.
\end{proposition}
\pf Suppose $p$ divides $|G|$, for some prime $p\ge 5$. Since $G$ is $\operatorname{NCI}$, $Z(G)$ is $\operatorname{NCI}$ by Lemma \ref{semi-rational-NCI-rational-center-closed-lemma} $(2)$. But since $G$ is nilpotent, $Z(G)$ contains an element of order $p$, which contradicts Corollary \ref{abelian-CCI-CI-FCCI-NCI-equivalence-corollary}. This proves the first part. The second part follows from Proposition \ref{NCI-direct-product-proposition}.\qed

\begin{corollary}
Let $G$ be an $\operatorname{NCI}$ $p$-group. Then $p\in\lbrace{2,3}\rbrace$.
\end{corollary}



\section*{Acknowledgement}
The first author is supported by the funding of UGC [NTA Ref. No. 211610129182], Govt. of India. 

\bibliographystyle{abbrv}
\bibliography{ref}

@book{isaacs1994character,
  title={Character theory of finite groups},
  author={Isaacs, I Martin},
  volume={69},
  year={1994},
  publisher={Courier Corporation}
}

@book{huppert2013endliche,
  title={Finite Groups I},
  author={Huppert, Bertram},
  year={2013},
  publisher={Springer-verlag}
}

@inproceedings{harary1974graphs,
  title={Which graphs have integral spectra?},
  author={Harary, Frank and Schwenk, Allen J},
  booktitle={Graphs and Combinatorics: Proceedings of the Capital Conference on Graph Theory and Combinatorics at the George Washington University June 18--22, 1973},
  pages={45--51},
  year={1974},
  organization={Springer}
}

@inproceedings{bachle2018integral,
  title={Integral group rings of solvable groups with trivial central units},
  author={B{\"a}chle, Andreas},
  booktitle={Forum Mathematicum},
  volume={30},
  number={4},
  pages={845--855},
  year={2018},
  organization={De Gruyter}
}

@article{babai1979spectra,
  title={Spectra of Cayley graphs},
  author={Babai, L{\'a}szl{\'o}},
  journal={Journal of Combinatorial Theory, Series B},
  volume={27},
  number={2},
  pages={180--189},
  year={1979},
  publisher={Elsevier}
}

@article{stein2007sage,
  title={Sage mathematics software},
  author={Stein, William},
  journal={http://www. sagemath. org/},
  year={2007},
  publisher={The Sage Group}
}

@article{ahmadi2009graphs,
  title={Graphs with integral spectrum},
  author={Ahmadi, Omran and Alon, Noga and Blake, Ian F and Shparlinski, Igor E},
  journal={Linear Algebra and its Applications},
  volume={430},
  number={1},
  pages={547--552},
  year={2009},
  publisher={Elsevier}
}

@article{abdollahi2009cayley,
  title={Which Cayley graphs are integral?},
  author={Abdollahi, Alireza and Vatandoost, Ebrahim},
  journal={The Electronic Journal of Combinatorics},
  volume={16},
  number={1},
  pages={R122},
  year={2009}
}

@article{klotz2010integral,
  title={Integral Cayley graphs over abelian groups},
  author={Klotz, Walter and Sander, Torsten},
  journal={The Electronic Journal of Combinatorics},
  pages={R81--R81},
  year={2010}
}

@article{alperin2012integral,
  title={Integral sets and Cayley graphs of finite groups},
  author={Alperin, Roger C and Peterson, Brian L},
  journal={The Electronic Journal of Combinatorics},
  pages={1--12},
  year={2012}
}

@article{huang2021integral,
  title={Integral and distance integral Cayley graphs over generalized dihedral groups},
  author={Huang, Jing and Li, Shuchao},
  journal={Journal of Algebraic Combinatorics},
  volume={53},
  number={4},
  pages={921--943},
  year={2021},
  publisher={Springer}
}

@article{cheng2019integral,
  title={Integral Cayley graphs over dicyclic group},
  author={Cheng, Tao and Feng, Lihua and Huang, Hualin},
  journal={Linear Algebra and its Applications},
  volume={566},
  pages={121--137},
  year={2019},
  publisher={Elsevier}
}

@article{ahmady2014integral,
  title={Integral Cayley graphs and groups},
  author={Ahmady, Azhvan and Bell, Jason P and Mohar, Bojan},
  journal={SIAM Journal on Discrete Mathematics},
  volume={28},
  number={2},
  pages={685--701},
  year={2014},
  publisher={SIAM}
}

@article{ma2016finite,
  title={On finite groups all of whose cubic Cayley graphs are integral},
  author={Ma, Xuanlong and Wang, Kaishun},
  journal={Journal of Algebra and Its Applications},
  volume={15},
  number={06},
  pages={1650105},
  year={2016},
  publisher={World Scientific}
}

@article{lu2018integral,
  title={Integral Cayley graphs over dihedral groups},
  author={Lu, Lu and Huang, Qiongxiang and Huang, Xueyi},
  journal={Journal of Algebraic Combinatorics},
  volume={47},
  pages={585--601},
  year={2018},
  publisher={Springer}
}

@article{cheng2023integral,
  title={Integral Cayley graphs over semi-dihedral groups},
  author={Cheng, Tao and Feng, Lihua and Yu, Guihai and Zhang, Chi},
  journal={Applicable Analysis and Discrete Mathematics},
  volume={17},
  number={2},
  pages={334--356},
  year={2023},
  publisher={JSTOR}
}

@article{godsil2025integral,
  title={Integral normal Cayley graphs},
  author={Godsil, Chris and Spiga, Pablo},
  journal={Journal of Algebraic Combinatorics},
  volume={62},
  number={1},
  pages={20},
  year={2025},
  publisher={Springer}
}

@article{guo2019integral,
  title={Integral Cayley graphs},
  author={Guo, W and Lytkina, Daria Viktorovna and Mazurov, Victor Danilovich and Revin, Danila Olegovich},
  journal={Algebra and Logic},
  volume={58},
  number={4},
  pages={297--305},
  year={2019},
  publisher={Springer}
}

@article{zieschang1988cayley,
  title={Cayley graphs of finite groups},
  author={Zieschang, Paul Hermann},
  journal={Journal of Algebra},
  volume={118},
  number={2},
  pages={447--454},
  year={1988},
  publisher={Elsevier}
}

@article{abdollahi2014groups,
  title={Groups all of whose undirected Cayley graphs are integral},
  author={Abdollahi, Alireza and Jazaeri, Mojtaba},
  journal={European Journal of Combinatorics},
  volume={38},
  pages={102--109},
  year={2014},
  publisher={Elsevier}
}

@article{foster2016spectra,
  title={Spectra of Cayley graphs of complex reflection groups},
  author={Foster-Greenwood, Briana and Kriloff, Cathy},
  journal={Journal of Algebraic Combinatorics},
  volume={44},
  number={1},
  pages={33--57},
  year={2016},
  publisher={Springer}
}

@article{balinska2002survey,
  title={A survey on integral graphs},
  author={Bali{\'n}ska, K and Cvetkovi{\'c}, Drag{\u{o}}s and Radosavljevi{\'c}, Zoran and Simi{\'c}, Slobodan and Stevanovi{\'c}, Dragan},
  journal={Publikacije Elektrotehni{\v{c}}kog fakulteta. Serija Matematika},
  pages={42--65},
  year={2002},
  publisher={JSTOR}
}

@article{abdollahi2011integral,
  title={Integral quartic Cayley graphs on abelian groups},
  author={Abdollahi, Alireza and Vatandoost, Ebrahim},
  journal={The Electronic Journal of Combinatorics},
  pages={P89--P89},
  year={2011}
}

@article{klotz2011integral,
  title={Integral Cayley graphs defined by greatest common divisors},
  author={Klotz, Walter and Sander, Torsten},
  journal={The Electronic Journal of Combinatorics},
  pages={P94--P94},
  year={2011}
}

@article{estelyi2014groups,
  title={On Groups all of whose Undirected Cayley Graphs of Bounded Valency are Integral},
  author={Est{\'e}lyi, Istv{\'a}n and Kov{\'a}cs, Istv{\'a}n},
  journal={The Electronic Journal of Combinatorics},
  pages={P4--45},
  year={2014}
}

@article{poddar2026algebraic,
  title={Algebraic degree of Cayley colour graphs},
  author={Poddar, Sauvik},
  journal={arXiv preprint arXiv:2602.08634},
  year={2026}
}

@article{conrad2020subgroup,
  title={Subgroup series II},
  author={Conrad, K.},
  journal={Lecture Notes available at https://kconrad. math. uconn. edu/blurbs/grouptheory/subgpseries2. pdf},
  year={2020}
}

@article{chen2014eigenvalues,
  title={On the eigenvalues of certain Cayley graphs and arrangement graphs},
  author={Chen, Bai Fan and Ghorbani, Ebrahim and Wong, Kok Bin},
  journal={Linear Algebra and its Applications},
  volume={444},
  pages={246--253},
  year={2014},
  publisher={Elsevier}
}

@article{chillag2010semi,
  title={Semi-rational solvable groups},
  author={Chillag, David and Dolfi, Silvio and others},
  journal={J. Group Theory},
  volume={13},
  number={4},
  pages={535--548},
  year={2010}
}

@article{ritter1990integral,
  title={Integral group rings with trivial central units},
  author={Ritter, J{\"u}rgen and Sehgal, Sudarshan K},
  journal={Proceedings of the American Mathematical Society},
  pages={327--329},
  year={1990},
  publisher={JSTOR}
}

@article{li2002isomorphisms,
  title={On isomorphisms of finite Cayley graphs—a survey},
  author={Li, Cai Heng},
  journal={Discrete mathematics},
  volume={256},
  number={1-2},
  pages={301--334},
  year={2002},
  publisher={Elsevier}
}

\end{document}